\documentclass[a4paper]{article}
\usepackage[affil-it]{authblk}
\usepackage[colorlinks,
             linkcolor=red,
             anchorcolor=blue,
             citecolor=blue
             ]{hyperref}

\usepackage{pdfsync}
\usepackage{amsmath}
\usepackage{amssymb}
\usepackage{amscd}
\usepackage{mathrsfs}
\usepackage{color}
\usepackage{ulem}
\usepackage{bm}
\usepackage{graphicx}
\usepackage{mathtools}
\usepackage{hyperref}

\renewcommand\footnotemark{}
\newtheorem{theorem}{Theorem}[section]
\newtheorem{proposition}[theorem]{Proposition}

\newtheorem{remark}[theorem]{Remark}

\newtheorem{definition}[theorem]{Definition}

\definecolor{highlight}{rgb}{.5,0,.5}

\def\I{\hbox{\bf I}}

\def\bX{{\mathbb X}}

\def\diag{\hbox{\rm diag}}

\def\GL{{\rm GL}}
\def\Gal{{\rm Gal}}

\def\stab{{\rm stab}}
\def\Zero{{\rm Zero}}

\def\Mat{{\rm Mat}}

\def\max{{\rm max}}

\def\I{{\mathcal I}}

\def\H{{\mathcal H}}

\def\bV{\mathbb{V}}

\def\bZ{\mathbb{Z}}
\def\bI{\mathbb{I}}

\def\F{{\mathcal F}}

\def\calH{{\mathcal H}}

\def\bfc{{\mathbf c}}
\def\bfd{{\mathbf d}}

\def\bfm{{\mathbf m}}

\def\bfZ{{\mathbf Z}}
\def\bfchi{{\bm \chi}}
\def\frakm{{\mathfrak{m}}}

\makeatletter

\newcommand{\Rmnum}[1]{\expandafter\@slowromancap\romannumeral #1@}
\makeatother

\begin{document}
\title{Note on the construction of Picard-Vessiot rings for linea differential equations}
\author{Ruyong Feng\footnote{This research was supported by NSFC under Grants No.11771433 and No.11688101, and Beijing Natural Science Foundation (Z190004).}}
\affil{KLMM, AMSS and UCAS,\\Chinese Academy of Sciences, Beijing 100190, China}
\date{}
\maketitle
\begin{abstract}
In this note, we describe a method to construct the Picard-Vessiot ring of a given linear differential equation.
\end{abstract}
\section{Introduction}
Throughout this note, $C$ stands for an algebraically closed field of characteristic zero and $k=C(t)$ denotes the ring of rational functions in $t$. We use $\delta$ to denote the usual derivation with respective to $t$. Then $k$ is a differential ring with derivation $\delta$. $X$ denotes $n\times n$ matrix $(X_{i,j})$ with indeterminate entries $X_{i,j}$.
Consider linear differential equation
\begin{equation}
\label{EQ:differentialeqn}
 \delta(Y)=AY
\end{equation}
where $A\in \Mat_n(k)$. Denote the ring over $k$ generated by the entries of $X$ and $1/\det(X)$ by $k[X,1/\det(X)]$. By setting $\delta(X)=AX$, we may endow $k[X,1/\det(X)]$ with a structure of differential ring that extends the differential ring $k$. An ideal $I$ of $k[X,1/\det(X)]$ is called a differential ideal (or $\delta$-ideal for short) if $\delta(I)\subset I$. Let $\frakm$ be a maximal $\delta$-ideal. Then the quotient ring
$$
   R=k(t)[X,1/\det(X)]/\frakm
$$
is the Picard-Vessiot ring of (\ref{EQ:differentialeqn}) over $k$. Recently, the construction of Picard-Vessiot rings receives many attentions. In this note, we shall describe a method to construct this ring, i.e. a method to compute a maximal $\delta$-ideal $\frakm$. Denote by $\Gal(R/k)$ the Galois group of (\ref{EQ:differentialeqn}) over $k$ that is defined to be the set of $k$-automorphisms of $R$ which commute with $\delta$. Let $\F$ be a fundamental matrix of (\ref{EQ:differentialeqn}) with entries in $R$. Then $\F$ induces an injective group homomorphism from $\Gal(R/k)$ to $\GL_n(C)$. The image of this homomorphism can be described via the stabilizer of the maximal $\delta$-ideal with $\F$ as a zero. 
\begin{definition}
\label{DEF:stabilizer}
For an ideal $I$ in $k[X,1/\det(X)]$, the stabilizer of $I$, denoted by $\stab(I)$, is defined to be
$$
   \{g\in \GL_n(C) | \forall\, P\in I, P(Xg)\in I\}.
$$
\end{definition}
It is well-known that $\stab(I)$ is an algebraic subgroup of $\GL_n(C)$. The following notion is introduced by Amzallag etc in \cite{amzallag-minchenko-pogudin}.
\begin{definition}
Let $G, H$ be two algebraic subgroups of $\GL_n(C)$. $H$ is called a toric envelope of $G$ if $H=T\cdot G$ (product as abstract groups) for some torus $T$.
\end{definition}
\begin{definition}
\label{DEF:protomaximalideals}
A radical $\delta$-ideal $I$ is called a toric maximal $\delta$-ideal of (\ref{EQ:differentialeqn}) over $k$ if there is a maximal $\delta$-ideal $\frakm$ containing $I$ such that $\stab(I)$ is a toric envelope of $\stab(\frakm)$.
\end{definition}
Roughly speaking, our method consists of the following three steps.
\begin{enumerate}
\item Compute a toric maximal $\delta$-ideal (\ref{EQ:differentialeqn}) over $k$, say $J$.
\item By $J$, compute a maximal $\delta$-ideal $I$ in $\bar{k}[X,1/\det(X)]$.
\item By $I$, compute a maximal $\delta$-ideal in $k[X,1/\det(X)]$.
\end{enumerate}

\section{Toric maximal $\delta$-ideals}
Let $\F$ be a fundamental matrix of (\ref{EQ:differentialeqn}) and $\nu$ a nonnegative integer. Denote by $\I(\F,\nu)$ the ideal in $k[X,1/\det(X)]$ generated by
\begin{equation}
\label{EQ:numaximalideals}
\{P\in k[X]_{\leq \nu} | P(\F)=0\}
\end{equation}
Remark that $\I(\F,\nu)$ is a $\delta$-ideal, in particular, $\I(\F,\infty)$ is a maximal $\delta$-ideal and $\stab(\I(\F,\infty))$ is the Galois group of (\ref{EQ:differentialeqn}) over $k$. For $S\subset k[X,1/\det(X)]$ and a field $K$, we shall use $\Zero(S)$ and $\bV_K(S)$ to denote the set of zeroes of $S$ in $\F\GL_n(C)$ and $\GL_n(K)$ respectively.
\begin{definition}
Let $G$ be an algebraic subgroup of $\GL_n(C)$ and $V$ an algebraic subvariety of $\GL_n(\bar{k})$. Assume that both $G$ and $V$ are defined over $k$. We call $V$ is a $G$-torsor over $k$ if $V=\alpha \cdot G(\bar{k})$ for some $\alpha\in V$.
\end{definition}
\begin{remark}
\label{RM:numaximalideal}
\begin{enumerate}
\item
One can easily verify that
$$
   \stab(\I(\F,\nu))=\{g\in\GL_n(C) | \forall\, P\in \I(\F,\nu), P(\F g)=0\}.
$$
Therefore $\F\stab(\I(\F,\nu))=\Zero(\I(\F,\nu))$.
\item For an ideal $I$ in $k[X,1/\det(X)]$, one can verify that $\stab(I)\subset \stab(\sqrt{I})$. Generally, the stabilizers of $I$ and $\sqrt{I}$ are not equal. For instance, set $n=2$ and
$$I=\langle x_{11}^2, x_{22}^2, x_{12}^3,x_{21}^3 \rangle.$$
Then $\sqrt{I}=\langle  x_{ij}: i,j=1,2\rangle.$  One has that $\stab(\sqrt{I})=\GL_2(C)$ but $\stab(I)=\{\diag(c_1,c_2)|c_1c_2\neq 0\}$. While for $I=\I(\F,\nu)$, an easy calculation yields that $\stab(I)=\stab(\sqrt{I})$.
\item Let $\mu$ be an integer greater than $\nu$. Then $\I(\F,\nu)\subset \I(\F,\mu)$ and From 1, one sees that $\stab(\I(\F,\mu))\subset \stab(\I(\F,\nu))$.
\end{enumerate}
\end{remark}
Let $\frakm$ be a maximal $\delta$-ideal and $G=\stab(\frakm)$. Theorem 1.28 on page 22 of \cite{vanderput-singer} states that $\bV_{\bar{k}}(\frakm)$ is a $G$-torsor over $k$. The following proposition implies that the similar property holds for $\bV_{\bar{k}}(\I(\F,\nu))$.
\begin{proposition}
\label{PROP:torsors}
$\bV_{\bar{k}}(\I(\F,\nu))$ is an $H$-torsor over $k$, where $H=\stab(\I(\F,\nu))$.
\end{proposition}
\begin{proof}
Let $K=\bar{k}(\F)$.
We first show that $\F H(K)=\bV_K(\I(\F,\nu))$. Let $S\subset C[X,1/\det(X)]$ be a finite set of polynomials defining $H$. Since $\I(\F,\nu)$ is generated by some polynomials with degree not greater than $\nu$, we may assume that elements of $S$ are of degree not greater than $\nu$. Then $\{P(\F^{-1}X)| P\in S\}$ defines $\F H(K)$ and it consists of polynomials with degree not greater than $\nu$. Let $J$ be the ideal in $k(\F)[X,1/\det(X)]$ generated by $S$. As $H$ contains the Galois group of (\ref{EQ:differentialeqn}) over $k$, one sees that $J$ is $\Gal(k(\F)/k)$-invariant. The proof of Lemma 1.29 on page 23 of \cite{vanderput-singer} implies that $J$ is generated by $J\cap k[X,1/\det(X)]$ and furthermore $J\cap k[X,1/\det(X)]$ is generated by some polynomials with degree not greater than $\nu$. Hence $J\cap k[X,1/\det(X)]\subset \I(\F,\nu)$. This implies that $\bV_K(\I(\F,\nu))\subset \F H(K)$. By Remark~\ref{RM:numaximalideal}, $\F H(K)\subset \bV_K(\I(\F,\nu))$. Therefore $\F H(K)=\bV_K(\I(\F,\nu))$. Suppose that $\alpha\in \bV_K(\I(\F,\nu))$. Then $\alpha=\F h$ for some $h\in H(K)$ and one has that 
$$\alpha H(K)=\F h H(K)=\F H(K)=\bV_K(\I(\F,\nu)).$$ From this, one sees easily that $\alpha H(\bar{k})=\bV_{\bar{k}}(\I(\F,\nu))$.
\end{proof}

\begin{definition}
\label{DEF:boundvarieties}
An algebraic subvariety $\bX$ of $\GL_n(C)$ is said to be bounded by $d$, where $d$ is a positive integer, if there are polynomials $f_1,\cdots,f_s\subset C[X]$ of degree at most $d$ such that
$$
   \bX=\GL_n(C)\cap \bV_C(f_1,\cdots,f_s).
$$
\end{definition}
Set
$$
\bfd(n)=\begin{cases}
               6, & n=2\\
               360, & n=3\\
               (4n)^{3n^2}, & n\geq 4
\end{cases}.
$$
\begin{theorem}[Theorems 3.1, 3.2 and 3.3 of \cite{amzallag-minchenko-pogudin}] Let $G\subset \GL_n(C)$ be a linear algebraic group. Then there exists a toric envelope of $G$ bounded by $\bfd(n)$.
\end{theorem}
\begin{proposition}
\label{PROP:torimaximals}
The radical of $\I(\F,\bfd(n))$ is a toric maximal $\delta$-ideal.
\end{proposition}
\begin{proof}
Set $J=\sqrt{\I(\F,\bfd(n))}$ and $\frakm=\I(\F,\infty)$. Then $J$ is a radical $\delta$-ideal and $\stab(\frakm)\subset \stab(\I(\F,\bfd(n)))=\stab(J)$. Let $H$ be a toric envelope of $\stab(\frakm)$ bounded by $\bfd(n)$, and denote
$$
   \tilde{J}=\{P\in k[X,1/\det(X)] | \forall\, h\in H, P(\F h)=0 \}.
$$
Then $\F H=\Zero(\tilde{J})$ and moveover Lemma 2.1 of \cite{feng} implies that $\tilde{J}$ is generated by some polynomials in $k[X]$ of degree not greater than $\bfd(n)$. In other words, $\tilde{J}\subset \I(\F,\bfd(n))$. One sees that
$$\Zero(\frakm)\subset \Zero(\I(\F,\bfd(n)))\subset \Zero(\tilde{J}).$$
By Remark~\ref{RM:numaximalideal}, we have that
$$
   \F \stab(\frakm) \subset \F \stab(\I(\F,\bfd(n)))\subset \F H
$$
which implies that $\stab(\frakm)\subset \stab(J)\subset H$. By Lemma 5.1 of \cite{amzallag-minchenko-pogudin}, $\stab(J)$ is also a toric envelope of $\stab(\frakm)$. So $J$ is a toric maximal $\delta$-ideal.
\end{proof}
The method described in Section 4.1 of \cite{feng} allows one to compute a $k$-basis of (\ref{EQ:numaximalideals}) i.e. a set of generators of $\I(\F,\nu)$. By Gr\"{o}bner bases computation one can compute the toric maximal $\delta$-ideal $\sqrt{\I(\F,\bfd(n))}$.
\section{Maximal $\delta$-ideals in $\bar{k}[X,1/{\rm det}(X)]$}
From now on, assume that we have already computed a set of generators of $\I(\F,\bfd(n))$.
Let $H=\stab(\I(\F,\bfd(n)))$ and $\alpha\in \bV(\I(\F,\bfd(n)))$. Then $\bV_{\bar{k}}(\I(\F,\bfd(n)))=\alpha H(\bar{k})$. Denote by $\langle \I(\F,\bfd(n))\rangle_{\bar{k}}$ the ideal in $\bar{k}[X,1/\det(X)]$ generated by $\I(\F,\bfd(n))$ and decompose $\sqrt{\langle \I(\F,\bfd(n)) \rangle_{\bar{k}}}$ into prime ideals:
$$
  \sqrt{\langle\I(\F,\bfd(n))\rangle_{\bar{k}}}=Q_1\cap \cdots \cap Q_s.
$$
Remark that the above decomposition can be done over the field $k(\alpha)$. Precisely, decompose $\sqrt{\langle \I(\F,\bfd(n))\rangle_{k(\alpha)}}$ into prime ideals in $k(\alpha)[X,1/\det(X)]$, say $\tilde{Q}_1,\cdots,\tilde{Q}_s$.
Then each $\bV(\tilde{Q}_i)$ is an $H^\circ$-torsor over $k(\alpha)$. Hence $\tilde{Q}_i$ generates a prime ideal in $\bar{k}[X,1/\det(X)]$ and then from these $\tilde{Q}_i$ we obtain $Q_i$. Without loss of generality, we assume that $\alpha \in \bV(Q_1)$. Denote $R=\bar{k}[X,1/\det(X)]/Q_1$. Remark that $R$ is a $\delta$-ring and all $Q_i$ are $\delta$-ideals. Let $\bI_{\bar{k}}(H^\circ)$ be the vanishing ideal of $H^\circ$ in $\bar{k}[X,1/\det(X)]$. Consider the map
\[
\begin{array}{cccc}
   \phi: & \bar{k}[X,1/\det(X)]/\bI_{\bar{k}}(H^\circ) &\longrightarrow & R \\[2mm]
                                        & P(X) &\longrightarrow & P(\alpha^{-1} X)
\end{array}.
\]
One sees that $\phi$ is an isomorphism of $\bar{k}$-algebras.
\begin{definition}
An element $h\in R$ is said to be hyperexponential over $\bar{k}$ if $h$ is invertible in $R$ and $\delta(h)=\lambda h$ for some $\lambda\in \bar{k}$. Suppose that $h_1$ and $h_2$ are hyperexponential over $\bar{k}$. We say $h_1$ and $h_2$ are similar if $h_1=rh_2$ for some $r\in \bar{k}$.
\end{definition}
The following proposition reveals the relation between hyperexponential elements of $R$ and characters of $H^\circ$.
\begin{proposition}
\label{PROP:relations}
If $\chi$ is a character of $H^\circ$, then $\chi(\alpha^{-1}\bar{X})$ is hyperexponential over $\bar{k}$ where $\bar{X}$ denotes the image of $X$ in $R$ under the natural homomorphism. Conversely, if $h\in R$ is hyperexponential over $\bar{k}$ then $h=r\chi(\alpha^{-1}\bar{X})$ for some character $\chi$ and some $r\in \bar{k}$.
\end{proposition}
\begin{proof}
Obviously, $\chi(\alpha^{-1}\bar{X})$ is invertible. We shall prove that $\chi(\alpha^{-1}\bar{X})$ is hyperexponential over $\bar{k}$. Set 
$r=\delta(\chi(\alpha^{-1}\bar{X}))|_{\bar{X}=\alpha}.$
We claim that $$\delta(\chi(\alpha^{-1}\bar{X}))=r\chi(\alpha^{-1}\bar{X}).$$
Denote
\begin{align*}
   P(\bar{X})&=\chi(\bar{X}^{-1}\alpha)\delta(\chi(\alpha^{-1}\bar{X}))\\
   &=\chi(\bar{X}^{-1}\alpha)\sum_{i,j}\frac{\partial \chi}{\partial X_{i,j}}(\alpha^{-1}\bar{X})\left(\delta(\alpha^{-1})\bar{X}+\alpha^{-1}A\bar{X}\right)_{i,j}
\end{align*}
where for a matrix $M$, $M_{i,j}$ denotes its $(i,j)$-entry. 
Since $\F$ is a zero of $\sqrt{\langle\I(\F,\bfd(n))\rangle_{\bar{k}}}$, there is $h\in H$ such that $\F h\in \bV_{\bar{k}(\F)}(Q_1)=\alpha H^{\circ}(\bar{k}(\F))$. Set $\bar{\F}=\F h$. Then $\bV_{\bar{k}(\F)}(Q_1)=\bar{\F}H^\circ(\bar{k}(\F))$. We can view $P(\bar{X})$ as a regular function on $\bV_{\bar{k}(\F)}(Q_1)$, and then view $P(\bar{\F}\bar{X})$ as a regular function on $H^\circ(\bar{k}(\F))$, i.e. $P(\bar{\F}\bar{X})\in \bar{k}(\F)[X,1/\det(X)]/\bI_{\bar{k}(\F)}(H^\circ)$. An easy calculation implies that for any $g\in H^\circ$,
$$
   \delta(\chi(\alpha^{-1}\bar{X}))|_{\bar{X}=\bar{\F}g}=\delta(\chi(\alpha^{-1}\bar{\F}g))=\delta(\chi(\alpha^{-1}\bar{\F})\chi(g))=\delta(\chi(\alpha^{-1}\bar{\F}))\chi(g).
$$
which implies that for any $g\in H^\circ$, one has that
\begin{align*}
 P(\bar{\F}g)-P(\bar{\F})&=\chi(\bar{X}^{-1}\alpha)\delta(\chi(\alpha^{-1}\bar{X}))|_{\bar{X}=\bar{\F}g}-P(\bar{\F})\\
 &=\chi(g^{-1}\bar{\F}^{-1}\alpha)\delta(\chi(\alpha^{-1}\bar{\F}))\chi(g)-P(\bar{\F})\\
 &=P(\bar{\F})-P(\bar{\F})=0.
\end{align*}
Hence the subvariety of $H^\circ(\bar{k}(\F))$ defined by $P(\bar{\F}\bar{X})-P(\bar{\F})$ contains $H^\circ$.  One can verify that the subvariety of $H^\circ(\bar{k}(\F))$ containing $H^\circ$ must be $H^\circ(\bar{k}(\F))$ itself. Thus $P(\bar{\F}\bar{X})$ is equal to $P(\bar{\F})$ and so is $P(\bar{X})$.  Now assigning $\bar{X}=\alpha$ in $P(\bar{X})$ yields that
$$
  P(\bar{\F})=P(\alpha)=\delta(\chi(\alpha^{-1}\bar{X}))|_{\bar{X}=\alpha}=r.
$$
This proves our claim.

Conversely, assume that $h\in R$ is hyperexponential over $\bar{k}$. Then $h$ is invertible in $R$, i.e. $h(\alpha \bar{X})$ is invertible in $\bar{k}[X,1/\det(X)]/\bI_{\bar{k}}(H^\circ)$. Equivalently, both $h(\alpha \bar{X})$ and $1/h(\alpha \bar{X})$ are regular functions on $H^\circ(\bar{k})$. Since $H^\circ(\bar{k})$ is a connected algebraic subgroup of $\GL_n(\bar{k})$. Rosenlicht's Theorem (see \cite{magid,singer} for instance) implies that $h(\alpha X)=r\chi(X)$ for some $r\in \bar{k}$ and some character $\chi$ of $H^\circ$. That is to say, $h(\bar{X})=r\chi(\alpha^{-1} \bar{X})$.
\end{proof}

Denote by $\bfchi(H^\circ)$ the group of characters of $H^\circ$ and assume that $\{\chi_1,\cdots,\chi_l\}$ is a set of generators of $\bfchi(H^\circ)$. By Proposition~\ref{PROP:torsors}, $\chi_i(\alpha^{-1}\bar{X})$ is hyperexponential over $\bar{k}$. Let $r_i$ be the certificate of $\chi_i(\alpha^{-1}\bar{X})$ for all $i=1,\cdots,l$. Set
$$
    \bfZ=\left\{ (m_1,\cdots,m_l)\in \bZ^l \left|  \exists\,f\in \bar{k}\setminus\{0\}\,s.t.\,\sum_{i=1}^l m_i r_i=\frac{\delta(f)}{f} \right.\right\}.
$$
Then $\bfZ$ is a finitely generated $\bZ$-module. Let $\{\bfm_1,\cdots,\bfm_\ell\}$ be a set of generator of $\bfZ$ and $f_i\in \bar{k}\setminus\{0\}$ satisfy that $\sum_{j=1}^l m_{i,j}r_j=\delta(f_i)/f_i$, where $\bfm_i=(m_{i,1},\cdots,m_{i,l})$. Let $\bar{\F}$ be as in the proof of Proposition~\ref{PROP:torsors}. Then $\bar{\F}$ is a fundamental matrix of (\ref{EQ:differentialeqn}) and for all $i=1,\cdots,\ell$,
$$
   \frac{\delta\left(\prod_{j=1}^l \chi_j(\alpha^{-1}\bar{\F})^{m_{i,j}}\right)}{\prod_{j=1}^l \chi_j(\alpha^{-1}\bar{\F})^{m_{i,j}}}=\frac{\delta(f_i)}{f_i}.
$$
Thus $\prod_{j=1}^l \chi_j(\alpha^{-1}\bar{\F})^{m_{i,j}}=c_if_i$ for some constant $c_i$ in $\bar{k}(\bar{\F})$ that is $C$. Now let $J$ be the radical ideal in $\bar{k}[X,1/\det(X)]$ generated by
$$
    Q_1\cup \left\{\left.\prod_{j=1}^l \chi_j(\alpha^{-1}X)^{m_{i,j}}-c_if_i \right| i=1,\cdots,\ell\right\}.
$$
\begin{proposition}
\label{PROP:algclosedmaximal}
$J$ is a maximal $\delta$-ideal in $\bar{k}[X,1/\det(X)]$.
\end{proposition}
\begin{proof}
Denote $\tilde{J}=\{P\in \tilde{k}[X,1/\det(X)] | P(\bar{\F})=0\}$ and $\tilde{G}=\stab(\tilde{J})$. Then $\tilde{J}$ is a maximal $\delta$-ideal containing $J$ and $\tilde{G}$ is the identity component of the Galois group $G$ of (\ref{EQ:differentialeqn}) over $k$. We shall prove that $J=\tilde{J}$. Let $\tilde{\alpha}\in \bV_{\bar{k}}(\tilde{J})$. Then by Proposition~\ref{PROP:torsors}, $\bV_{\bar{k}}(\tilde{J})=\tilde{\alpha} \tilde{G}(\bar{k})$ and $\bV_{\bar{k}}(Q_1)=\tilde{\alpha}H^\circ(\bar{k})$. So $\tilde{G}\subset H^\circ$. Because $H$ is a toric envelope of $G$, one sees that $H^\circ$ is a toric envelope of $\tilde{G}$  By Lemma 4.1 of \cite{amzallag-minchenko-pogudin} and  Proposition 2.6 of \cite{feng2}, $\tilde{G}$ is defined by some characters of $H^\circ$, i.e. $\tilde{G}=\cap_{\chi\in S} \ker(\chi)$ where $S$ is a finite subset of $\bfchi(H^\circ)$. Suppose that $\chi\in S$. Since $\tilde{\alpha}^{-1}\bar{\F}\in \tilde{G}(\bar{k}(\bar{\F}))$, $\chi(\tilde{\alpha}^{-1}\bar{\F})=1$. So $\chi(\alpha^{-1}\bar{\F})=\chi(\alpha^{-1}\tilde{\alpha})$. Write $\chi=\prod_{i=1}^l \chi_i^{d_i}$ with $d_i\in \bZ$. One then has that
\begin{align*}
  \frac{\delta\left(\prod_{i=1}^l \chi_i(\alpha^{-1}\bar{\F})^{d_i}\right)}{\prod_{i=1}^l \chi_i(\alpha^{-1}\bar{\F})^{d_i}}=\sum_{i=1}^l d_i\frac{\delta\left(\chi_i(\alpha^{-1}\bar{\F})\right)}{\chi_i(\alpha^{-1}\bar{\F})}
  =\sum_{i=1}^l d_i r_i=\frac{\delta\left(\chi(\alpha^{-1}\tilde{\alpha})\right)}{\chi(\alpha^{-1}\tilde{\alpha})}.
\end{align*}
This implies that $(d_1,\cdots,d_l)\in \bfZ$. Write $(d_1,\cdots,d_l)=\sum_{j=1}^\ell s_j \bfm_j$, where $s_j\in \bZ$. Now for any $g\in \bV_{\bar{k}}(J)$, one has that $\prod_{j=1}^l \chi_j(\alpha^{-1}g)^{m_{i,j}}=c_if_i$ for all $i=1,\cdots,\ell$ and then
\begin{align*}
  \chi(\alpha^{-1}g)&=\prod_{j=1}^l \chi_j(\alpha^{-1}g)^{d_i}=\prod_{j=1}^l \chi_j(\alpha^{-1}g)^{\sum_{i=1}^\ell s_i m_{i,j} }\\
   &=\prod_{i=1}^\ell \left(\prod_{j=1}^l \chi_j(\alpha^{-1}g)^{m_{i,j}}\right)^{s_i}=\prod_{i=1}^\ell (c_if_i)^{s_i}.\\
\end{align*}
In particular, one has that $\chi(\alpha^{-1}\tilde{\alpha})=\prod_{i=1}^\ell (c_if_i)^{s_i}$.
On the other hand, for any $g\in \bV_{\bar{k}}(J)$, one has that $g=\tilde{\alpha}\tilde{g}$ for some $\tilde{g}\in H^\circ(\bar{k})$ and then
$$
 \chi(\alpha^{-1}\tilde{\alpha})=\prod_{i=1}^\ell (c_if_i)^{s_i}= \chi(\alpha^{-1}g)=\chi(\alpha^{-1}\tilde{\alpha}\tilde{g})=\chi(\alpha^{-1}\tilde{\alpha})\chi(\tilde{g}).
$$
This implies that $\chi(\tilde{g})=1$. By the choice of $\chi$, one sees that $\tilde{g}\in \tilde{G}(\bar{k})$. Consequently, $\bV_{\bar{k}}(J)\subset \tilde{\alpha}\tilde{G}(\bar{k})=\bV_{\bar{k}}(\tilde{J})$. Since $J$ is radical, $\tilde{J}\subset J$. Thus $J=\tilde{J}$.
\end{proof}

 Given $\calH=\{\chi_1(\alpha^{-1}\bar{X}), \cdots, \chi_l(\alpha^{-1}\bar{X})\}$, a method described in \cite{compoint-singer} allows one to compute a set of generators of $\bfZ$, say $\{\bfm_1,\cdots,\bfm_\ell\}$ and the corresponding $f_1,\cdots,f_\ell$. Thus in order to compute the maximal $\delta$-ideal $J$, it suffices to compute $\calH$. Proposition~\ref{PROP:relations} indicates that $\calH$ can be obtained via computing suitable hyperexponential elements in $R$. Remark that for any two $\chi_1, \chi_2\in \bfchi(H^\circ)$, $\chi_1(\alpha^{-1}\bar{X})$ and $\chi_2(\alpha^{-1}\bar{X})$ are similar if and only if $\chi_1=\chi_2$. To see this, assume that $\chi_1(\alpha^{-1}\bar{X})=r\chi_2(\alpha^{-1}\bar{X})$ for some $r\in \bar{k}$. Then taking $\bar{X}=\alpha$, one sees that $r=1$ and thus $\chi_1=\chi_2$. Hence elements in $\H$ are not similar to each other and take value 1 at $\bar{X}=\alpha$. Furthermore, by Proposition B.17 of \cite{feng}, there are generators of $\bfchi(H^\circ)$ that can be represented by polynomials in $C[X]$ with degree not greater than
 $$
    \kappa=(2n)^{3\cdot8^{n^2}}{n^2+(2n)^{3\cdot8^{n^2}} \choose n^2} \max_i\left\{{{n^2+(2n)^{3\cdot8^{n^2}} \choose n^2} \choose i}^2\right\}.
 $$
 A small modification of the method developed in \cite{feng2} enables us to compute $\calH$. Denote $R_{\leq \kappa}=\bar{k}[X]_{\leq \kappa}/(Q_1\cap \bar{k}[X])_{\leq \kappa}$. Then $R_{\leq \kappa}$ is a $\bar{k}$-vector space of finite dimension. Moreover $R_{\leq \kappa}$ is a $\delta$-vector space, i.e. $\delta(R_{\leq \kappa})\subset R_{\leq \kappa}$. Assume that $\{P_1,\cdots,P_s\}$ is a basis of $R_{\leq \kappa}$ where $s=\dim(R_{\leq \kappa})$. Then there is a $s\times s$ matrix $B$ with entries in $\bar{k}$ such that
 \[
    \delta\begin{pmatrix} P_1 \\ \vdots \\ P_s\end{pmatrix}=B\begin{pmatrix} P_1 \\ \vdots \\ P_s\end{pmatrix}.
 \]
 Suppose that $h=\sum_{i=1}^s c_i P_i$ with $c_i\in \bar{k}$ is hyperexponential over $\bar{k}$, i.e. $\delta(h)=rh$ for some $r\in \bar{k}$. An easy calculation yields that
 \[
      \delta\begin{pmatrix} c_1 \\ \vdots \\ c_s \end{pmatrix}-r\begin{pmatrix} c_1 \\ \vdots \\ c_s \end{pmatrix}=-B^t\begin{pmatrix} c_1 \\ \vdots \\ c_s \end{pmatrix}
 \]
 where $B^t$ denotes the transpose of $B$. Let $\tilde{h}$ be a hyperexponential element over $\bar{k}$ satisfying that $\delta(\tilde{h})=-r\tilde{h}$ and $\bfc=(c_1,\cdots,c_s)^t$. Then $\bfc \tilde{h}$ is a hyperexponential solution of $\delta(Y)=-B^tY$. The algorithm developed in \cite{singer2} allows us to compute a $C$-basis of the solution space of $\delta(Y)=-B^t Y$ that consists of hyperexponential solutions, say $\bfc_1 h_1,\cdots, \bfc_d h_d$. Write $\bfc_i=(c_{i,1},\cdots,c_{i,s})$. Then $\sum_{j=1}^s c_{i,j}P_j$ is a hyperexponential element in $R$ and one easily sees that $\sum_{j=1}^s c_{i,j}P_j$ and $\sum_{j=1}^s c_{i',j}P_j$ are similar if and only if $\bfc_i=r\bfc_{i'}$ for some nonzero $r\in \bar{k}$. Without loss of generality, we may assume that $\{\bfc_1,\cdots,\bfc_{d'}\}$ is a $\bar{k}$-basis of the vector space over $\bar{k}$ spanned by $\bfc_1,\cdots,\bfc_d$. Set $\bar{h}_i=\sum_{j=1}^s c_{i,j}P_j$ for all $i=1,\cdots,d'$. Multiplying a suitable element in $\bar{k}$, we may assume that $\bar{h}_i$ takes value 1 at $\bar{X}=\alpha$. Then $\{\bar{h}_1,\cdots,\bar{h}_{d'}\}$ is the required set.
 \begin{remark}
 One has that for a character $\chi$ of $H^\circ$, $\chi(\alpha^{-1}\bar{X})$ is actually a hyperexponential element over $k(\alpha)$ in $k(\alpha)[X,1/\det(X)]/\tilde{Q}_1$. So in practice computation, one only need to compute solutions that are hyperexponential over $k(\alpha)$.
 \end{remark}

 \section{Maximal $\delta$-ideals in $k[X,1/{\rm det}(X)]$}
 The previous section enables us to compute a maximal $\delta$-ideal in $\bar{k}[X,1/\det(X)]$. Now assume that we have such a maximal $\delta$-ideal $\bar{I}$ in hand. Suppose that $\bar{I}$ is generated by a finite set $S\subset \bar{k}[X,1/\det(X)]$. Let $\tilde{k}$ be a finite Galois extension of $k$ such that $S\subset \tilde{k}[X,1/\det(X)]$. We define the action of $\Gal(\tilde{k}/k)$ on an element $f\in S$ to be the action of $\Gal(\tilde{k}/k)$ to the coefficients of $f$. Assume that $\{\rho_1(S),\cdots,\rho_d(S)\}$ is the orbit of $S$ under the action of $\Gal(\tilde{k}/k)$. Denote
 $$
     \frakm=k[X,1/\det(X)]\cap _{i=1}^d \langle \rho_i(S)\rangle_{\tilde{k}}
 $$
 where $\langle \rho_i(S) \rangle_{\tilde{k}}$ denotes the ideal in $\tilde{k}[X,1/\det(X)]$ generated by $\rho_i(S)$. We claim that $\frakm$ is a maximal $\delta$-ideal in $k[X,1/\det(X)]$. Since $\cap_{i=1}^d\langle \rho_i(S)\rangle_{\tilde{k}}$ is invariant under the action of $\Gal(\tilde{k}/k)$, one has that
 $$
    \langle \frakm \rangle_{\tilde{k}}=\cap_{i=1}^d\langle \rho_i(S)\rangle_{\tilde{k}}.
 $$
 Let $I$ be a maximal $\delta$-ideal in $k[X,1/\det(X)]$ containing $\frakm$. Decompose $\langle I \rangle_{\tilde{k}}$ into prime ideals in $\tilde{k}[X,1/\det(X)]$:
 $$
    \langle I \rangle_{\tilde{k}}=Q_1\cap \cdots \cap Q_s.
 $$
 Since $\langle I \rangle_{\tilde{k}}$ is a $\delta$-ideal, all $Q_i$ are $\delta$-ideals. Furthermore, one can verify that $\{Q_1,\cdots,Q_s\}$ is the orbit of $Q_1$ under the action of $\Gal(\tilde{k}/k)$. Because  $\langle \frakm \rangle_{\tilde{k}}\subset \langle I \rangle_{\tilde{k}}$, 
 there exists $\rho\in \Gal(\tilde{k}/k)$ such that $\langle \rho(S) \rangle_{\tilde{k}}\subset Q_1$. Remark that $\langle \rho(S)\rangle_{\tilde{k}}$ is a maximal $\delta$-ideal as so is $\langle S \rangle_{\tilde{k}}$. Thus $\langle \rho(S) \rangle_{\tilde{k}}=Q_1$. Therefore
 $$
     \langle \frakm \rangle_{\tilde{k}}=\cap_{i=1}^d\langle \rho_i(S)\rangle_{\tilde{k}}=Q_1\cap \cdots \cap Q_s= \langle I \rangle_{\tilde{k}}
 $$
 which implies that $\frakm=I$, because $\frakm=\langle \frakm \rangle_{\tilde{k}}\cap k[X,1/\det(X)]$ and $I=\langle I \rangle_{\tilde{k}}\cap k[X,1/\det(X)]$. This proves our claim.
\bibliographystyle{plain}

 \end{document}